\documentclass[12pt,reqno]{amsart}

\usepackage{amsfonts}
\usepackage{amssymb}
\usepackage{amsmath}
\usepackage{amsthm}
\usepackage{bbm}
\usepackage{graphics}
\usepackage{mathrsfs}
\usepackage{float}

\newtheorem{lem}{Lemma}
\newtheorem{theorem}{Theorem}
\newtheorem{cor}{Corollary}

\newtheorem{remark}{Remark}

\usepackage{fullpage}
\begin{document}
\title{On a Subset Metric}

\bibliographystyle{plain}
\author{Richard Castro}
\address{
Richard Castro\\
Department of Mathematics and Statistics \\
San Diego State University\\
5500 Campanile Drive, San Diego, 92182, CA,  USA}

\email{rcastro0899@sdsu.edu}

\author{Zhibin Chang}
\address{
Zhibin Chang\\
Department of Mathematics and Statistics \\
San Diego State University\\
5500 Campanile Drive, San Diego, 92182, CA,  USA}

\email{zchang@sdsu.edu}

\author{Ethan Ha}
\address{
Ethan Ha\\
Department of Mathematics and Statistics \\
San Diego State University\\
5500 Campanile Drive, San Diego, 92182, CA,  USA}

\email{ethankaweiha@gmail.com}

\author{Evan Hall}
\address{
Evan Hall\\
Department of Mathematics and Statistics \\
San Diego State University\\
5500 Campanile Drive, San Diego, 92182, CA,  USA}

\email{elhall@sdsu.edu }

\author{Hiren Maharaj}
\address{
Hiren Maharaj\\
Department of Mathematics and Statistics \\
San Diego State University\\
5500 Campanile Drive, San Diego, 92182, CA,  USA}

\email{hmaharaj@sdsu.edu}

\begin{abstract}
 For a bounded metric space $X$, we define a metric on the set of all finite subsets of $X$.  This generalizes the  sequence-subset distance introduced by Wentu Song, Kui  Cai and  Kees A. Schouhamer Immink \cite{cai} to study error correcting codes for DNA based data storage.    This work also complements the work of Eiter and Mannila \cite{EiterMannila} where they study extensions of distance functions to subsets of a space in the context of various applications.
\end{abstract}

\maketitle

%
%
%

\section{Introduction}\label{sec1}

To design error correcting codes for DNA storage channels, a new metric, called the sequence-subset distance,  was introduced
in \cite{cai}. This metric generalizes the Hamming distance to a distance function defined between any two sets of unordered vectors.  
 The definition is as follows. Let $\mathbb{A}$ be a fixed finite alphabet and $L\ge 1$ an integer. For any $x_1, x_2 \in \mathbb{A}^L$, 
the Hamming distance $d_H(x_1, x_2)$ between $x_1$ and $x_2$ is the number of coordinates in which $x_1$ and $x_2$ differ. For two subsets $X_1, X_2 \subset \mathbb{A}^L$, with $\vert X_1\vert  \le \vert X_2\vert $, and any injection $\chi: X_1 \to X_2$, the $\chi-$ distance between $X_1$ and $X_2$ is defined to be 
\begin{equation}\label{caiequation}
d_\chi(X_1, X_2) = \sum_{x\in X_1} d_H(x, \chi(x))  + L(\vert X_2\vert -\vert X_1\vert ).
\end{equation} The sequence-subset distance between $X_1$ and $X_2$ is defined to be 
\begin{equation*}
d_S(X_1, X_2) = d_S(X_2,X_1) = \min\{ d_\chi(X_1, X_2) \vert  \chi: X_1 \to X_2  \text{ is an  injection} \}.
\end{equation*}
In \cite{cai} it is shown that $d_S$ is in fact a metric on the set of subsets of $\mathbb{A}^L$.

In this note we generalize the sequence-subset distance as follows. 
Let $X$ be a  bounded metric space. For each $y \in X$, let 
 $M: X\to \mathbb{R}$ be a function such that 
\begin{equation}\label{condition}
d(x,y) \le M(x) \le d(x,z) + M(z)
\end{equation} for all $x,y, z\in X$. 
Put $Y := \mathcal{F}(X)$, the set of all finite subsets of $X$. 
 For $A,B \in Y$, with $\vert A\vert  \leq \vert B\vert $, and any injection $\chi: A \to  B$, the $\chi-$ distance between
 $A$ and $B$ to defined to be  $$d_{\chi}(A,B) :=  \sum_{x \in A} d(x, \chi(x)) + \sum_{y\in B\setminus \chi(A)} M(y).$$ Now the distance between $A$ and $B$ is  defined to be  \begin{equation}\label{mainmetric}
 d_S(A,B) = d_S(B,A) :=  \min\{d_{\chi}(A,B)\vert \ \chi: A \to  B\ \text{is an injection}\}.
 \end{equation}
We show in Section \ref{sec2} that $d_S$ is indeed a metric on $\mathcal{F}(X)$. We will refer to this distance function simply as a subset metric. 

There is some flexibility in the choice of the function $M$.
Since $X$ is a bounded metric space, we can select the function $M$ to have constant value $D := \sup\{ d(x,y): x,y\in X\}$. 
In the case of the Hamming metric $d= d_H$ on $X= \mathbb{A}^L$, this is tantamount to choosing $M(y)$ 
to be the constant $L$ for all $y\in X$ and the subset-sequence metric of \cite{cai} is recovered.  In fact $M$ could be be any constant valued function whose value is an upper bound for the metric $d$ on $X$. 
Alternatively, one could define $M$ as follows: for each $x\in X$, let 
\begin{equation}\label{lowM}
M(x) = \sup\{d(x,y): y\in X\}.
\end{equation} Condition (\ref{condition}) is satisfied: 
for all $y\in X$,   $d(x,y) \le d(x,z) + d(z,y) \le d(x,z) + M(z)$ whence $M(x) \le d(x,z)  + M(z)$.

As for the sequence-subset distance of  \cite{cai}, the subset distance between $A$ and $B$ can be computed from a minimum weight perfect matching of the bipartite graph whose partite sets are $A$ and $B$;  the edge joining $a\in A$ with $b\in B$ is assigned weight $d(a,b)$. The  Kuhn-Munkres algorithm does this in time $O(\vert B\vert ^3)$ \cite{munkres}. 

The generalized metric could potentially have more applications. For example, take $X$ to   be the vertex set of a finite connected graph and $d(x,y)$ the length of the shortest  path between $x$ and $y$. Then $d_S$ is a metric on the power set $2^X$ and provides a measure of  distance between collections of vertices. 

Another example is image recognition. In this case take $X$ to be  a bounded subset of the standard Euclidean plane (for example, corresponding to a raster of pixels). For simplicity we take $X=[0,1]\times[0,1]$ the unit square as an example and $d(p,q) = \vert \vert p-q\vert \vert $ is the standard Euclidian distance. Each finite subset of $X$ would correspond to an image. Using (\ref{lowM}) to define the function 
$M(p)$,  we have  
$M(p) : = \max\{\vert \vert p-c_1\vert \vert , \vert \vert p-c_2\vert \vert , \vert \vert p-c_3\vert \vert , \vert \vert p-c_4\vert \vert  \}$ where $c_1, c_2, c_3, c_4$ are the four corners of $X$. Alternatively,   $M$ could be replaced by the constant function whose value is  $D= \sqrt{2}$.  

Distance functions between subsets of a metric space and also measure spaces have been widely studied, see \cite{conci} for a survey of such distances; see also \cite{ency}. One of the most widely used subset metrics is the Hausdorff metric \cite{conci}. This metric has many variations, but we state
one version for comparison. Let $X$ be a bounded metric space with metric $d$. For non-empty compact subsets $A, B$ of $X$, 
define $$h(A,B) := \max\{ \max_{a\in A} d(a,B), \max_{b\in B} d(b,A) \}$$
where $d(a,B) := \min_{a\in A} d(x,a)$ and $d(b,A)$ is defined likewise. The function $h$ gives a metric on the set of all compact subsets of 
$X$
that generalizes $d$: $h(\{a\},\{b\}) = d(a,b)$ for all $a,b\in X$.
If $X$ is finite, the Hausdorff metric is computable in polynomial time and does have theoretical benefits, for example, it is complete if $X$ is complete with respect to $d$.  However, as pointed out in 
\cite{EiterMannila}, it may not be appropriate for some applications since the metric does not take into 
account the entire configuration of some finite sets. On the other hand, the subset-sequence metric formulated in  \cite{cai} for the purpose of comparing  of  DNA sequences
provides a finer comparison between two collections of sequences and is thus  a more appropriate distance measure in that situation. Each term involving $L$ on the right side of  (\ref{caiequation}) expresses a natural worst case weight for a DNA strand that is too far away from the other set. 
While the authors of this work were primarily motivated by generalizing the work of \cite{cai}, this work also complements that of 
\cite{EiterMannila} where they study extensions of distance measures to subsets more generally. For comparison, we briefly recall some of the main results from \cite{EiterMannila}. A distance function  $\Delta$ on a non-empty set $B$ is one that satisfies all of the axioms to be a metric, except possibly the triangle inequality. 
 In \cite{EiterMannila}, the authors consider the problem of extending a distance function to the set of non-empty finite subsets of $B$. They also discuss algorithms for computing such extensions. 
 To measure a distance between two non-empty subsets $S_1, S_2$ of $B$, they discuss four distance functions: the sum of minimum distances \cite{Niini} $$d_{md}(S_1, S_2) := \frac{1}{2} \left ( \sum_{e\in S_1} \Delta(e, S_2) +  \sum_{e\in S_2} \Delta(e, S_1)  \right ),$$ 
the surjective distance $$d_s(S_1, S_2) :=  \min_\eta \sum_{(e_1, e_2) \in \eta } \Delta(e_1, e_2) $$ where the minimum is over all surjections $\eta$ from the larger set  to the smaller set  (due to  G. Oddie in \cite{Oddie}), the Fair surjection distance 
$$d_{fs}(S_1, S_2) :=  \min_\eta \sum_{(e_1, e_2) \in \eta } \Delta(e_1, e_2) $$ where the minimum is over all {\em fair} surjections $\eta$ from the larger set  to the smaller set (a surjection $\eta:S_1 \to S_2$ is  called fair if $\left \lvert \vert \eta^{-1}(x)\vert  - \vert  \eta^{-1}(y)\vert  \right \rvert  \le 1$ for all $x,y\in S_1$; this is also due to G. Oddie in \cite{Oddie}) and they introduce the Link distance 
$$d_l(S_1, S_2) :=  \min_R \sum_{(e_1, e_2) \in R } \Delta(e_1, e_2) $$ where the minimum is over all linking relations $R$ between $S_1$ and $S_2$ (a subset $R\subset S_1\times S_2$ is called a linking relation if for all $e_1\in S_1$, there exists $e_2\in S_2$ such that $(e_1, e_2) \in R$ and also if for all $e_2\in S_2$, there exists $e_1\in S_1$ such that $(e_1, e_2) \in R$).   While they show that these distance functions fail to be a metric in the case that $B$ is a finite subset of the integral plane and $\Delta$ is  the Manhattan metric,  
Eiter and Mannila  present an elegant construction, called the  metric infimum method, that produces a metric $\Delta^\omega$ from a given distance function $\Delta$. Interestingly, they demonstrate that $d_s^\omega = d_{fs}^\omega = d_l^\omega$. The authors in \cite{EiterMannila} argue that the link metric is very intuitive in some contexts. It would interesting to also study this metric in the context of error correcting codes for DNA data storage.  

The rest of the paper is devoted to proving that (\ref{mainmetric}) is indeed a metric.

\section{Proofs}\label{sec2}

Thoughout this section $X$ is a bounded metric space with metric $d$,  the  function $M:X\to \mathbb{R}$ is one that satisfies the condition 
(\ref{condition}), $d_S$ is the function defined  by (\ref{mainmetric}) and  $\mathcal{F}(X)$
is the set of all finite subsets of $X$. 
In this section we prove that the  function $d_S$ is  a metric on $\mathcal{F}(X)$.
While the main steps followed here are inspired by \cite{cai}, there are differences to account for the presence of the function $M$ in the definition of $d_S$.

\begin{lem} \label{lemma1} For any $X_1,X_2 \in \mathcal{F}(X)$, such that $\vert X_1\vert  \leq \vert X_2\vert $, there exists an injection $\chi_0: X_1 \to  X_2$, such that $d_S(X_1,X_2) = d_{\chi_0}(X_1,X_2)$ and $\chi_0(x) = x$ for all $x \in X_1 \cap X_2$.
\end{lem}

\begin{proof}
If $X_1 \cap X_2 = \emptyset$, then the statement is vacuously true. Suppose that $X_1\cap X_2 \ne \emptyset$.
Choose $\chi:X_1\to  X_2$ such that $d_S(X_1,X_2) = d_\chi(X_1,X_2)$. 
The proof will be in two parts. First we show that, if necessary,  $\chi$ can be redefined on $X_1 \cap X_2$ so that $d_S(X_1, X_2) = d_\chi(X_1,X_2)$ and
 $X_1\cap X_2$ is contained in   the image of $\chi$. Next we will show that $\chi$ can be further adjusted to have the desired properties.
 
 Suppose
 that some $x_0 \in X_1 \cap X_2$ does  not belong to the image of $\chi$. Then we  redefine $\chi$ at $x_0$  to form a new embedding
$\nu:  X_1 \to  X_2$ by setting 
$$\nu(x) = \left \{  \begin{array}{ll} \chi(x) & \hbox{ if } x\ne x_0 \\
x_0 & \hbox{ if } x= x_0. \end{array} \right. $$
By definition $d_S(X_1, X_2) \le d_\nu(X_1,X_2)$. 
Note that $\nu(X_1) = \left ( \chi(X_1) \setminus \{ \chi(x_0) \} \right )  \cup \{ x_0 \}$ and 
\begin{equation}\label{secondhalf}
\sum_{x\in X_1} d(x, \nu(x)) =  \sum_{x\in X_1} d(x, \chi(x))   \, \,   - d(x_0, \chi(x_0)).
\end{equation}
Since $x_0\in X_2 \setminus \chi(X_1)$,  $\chi(x_0) \not \in X_2 \setminus \chi (X_1)$ and $\chi(x_0) \in  X_2\setminus \nu(X_1)$, it follows that 
\begin{equation}\label{firsthalf}
\sum_{y \in X_2\setminus \nu(X_1) } M(y) = \sum_{y\in X_2\setminus \chi(X_1) } M(y)  - M(x_0)  + M(\chi(x_0)).
\end{equation}
Combining (\ref{secondhalf}) and (\ref{firsthalf}), we get that 
$$d_\nu(X_1,X_2)  = d_\chi(X_1, X_2) +M(\chi(x_0)) - M(x_0) - d(x_0, \chi(x_0)).$$
From the condition (\ref{condition}), it follows that $d_\nu(X_1,X_2) \le d_\chi(X_1,X_2) = d_S(X_1, X_2)$.
Thus $d_S(X_1, X_2) = d_\nu(X_1, X_2)$ and $\nu(x_0) = x_0$. By repeatedly applying the above procedure we will obtain an embedding of $X_1$ into $X_2$, which we also call $\chi$, with the property that 
 $X_1\cap X_2 \subseteq Im(\chi)$.

Let $x_1\in X_1 \cap X_2$. Next we  show that if $\chi(x_1) \ne x_1$ then  we can adjust the embedding $\chi$ 
to form a new embedding $\mu:X_	1\to  X_2$ such that we have $\mu(x_1)=x_1$ and still have that $d_S(X_1,X_2) = d_\mu(X_1,X_2)$. 
 From above we know that there exists $z\in X_1$ such that 
$\chi(z) = x_1$. Put $y= \chi(x_1)$ and 
define
$$\mu(x) = \left \{  \begin{array}{ll} \chi(x) & \hbox{ if } x\ne x_1, z \\
x_1 & \hbox{ if } x= x_1 \\
y &   \hbox{ if } x= z. 
\end{array} \right. $$

Then $\mu : X_1 \to  X_2$ is an injection and, by the definition of the subset distance,  $d_S(X_1,X_2) \le d_\mu(X_1,X_2)$.
Also we have that 
\begin{eqnarray*}
d_\chi (X_1,X_2)  & = &    d(x_1,y) + d(z,x_1) + \left (   d_\mu (X_1,X_2) -d(x_1,x_1)-d(z,y)  \right)  \\
& = &  d_\mu (X_1,X_2) + d(x_1,y) + d(z,x_1)-d(z,y) \\
&\ge &  d_\mu (X_1,X_2) 
\end{eqnarray*} where the last inequality follows from the triangle inequality.  Thus $d_S(X_1, X_2)  \ge d_\mu(X_1,X_2)$
and we see that $d_S(X_1,X_2) = d_{\chi}(X_1,X_2) = d_\mu(X_1,X_2)$ and $\mu(x_1) = x_1$.
By repeated application of the above procedure, we obtain an embedding with the desired property. 
\end{proof}

\begin{cor}
 For any $X_1, X_2 \in \mathcal{F}(X)$, $$d_S(X_1,X_2) = d_S(X_1 \setminus X_2, X_2 \setminus X_1).$$
\end{cor}
\begin{proof}
This is a direct consequence of Lemma \ref{lemma1}  and the definition of $d_{\chi}(\cdot,\cdot)$.
\end{proof}

\begin{lem} \label{simple}
Suppose that $X_1, X_2  \in \mathcal{F}(X)$ with $\vert X_1\vert \le \vert X_2\vert $. Then for any $b\in X$, 
$d_S(X_1,X_2) \le d_S(X_1, X_2\cup \{ b \})$. 
\end{lem}
\begin{proof}
Suppose $\chi: X_1\to  X_2\cup\{b\}$ such that $d_S(X_1, X_2\cup\{b\}) = d_\chi(X_1, X_2\cup \{b\} )$.
If $\chi(X_1) \subseteq X_2$, then $d_\chi(X_1, X_2\cup \{ b \} ) = d_\chi(X_1, X_2) + M(b) \ge d_S(X_1,X_2)  + M(b) \ge d_S(X_1, X_2)$. 
If $\chi(X_1) \not \subset X_2$, then $\chi(a) =b$ for some $a\in X_1$ and $\vert X_2\vert > \vert X_1\vert $. Fix $c \in X_2\setminus \chi(X_1)$
and define $\eta: X_1 \to X_2$ by
$$
\eta(x) = \left \{ \begin{array}{ll} 
\chi(x) & \hbox{if } x\ne a \\
c & \hbox{if } x =a.
\end{array} \right.
$$ 
Then
$\eta(X_1) = \left (  \chi(X_1)\setminus \{ b\} \right ) \cup \{ c \} $ so 
$X_2 \cup \{ b \} \setminus \chi(X_1)$ is the disjoint union $ \left ( X_2 \setminus \eta(X_1) \right )  \cup \{ c \}$
 and
\begin{eqnarray*}
& & d_S(X_1, X_2\cup \{b\} ) \\
&=& d_\chi(X_1, X_2\cup \{b\}) \\
&=& \sum_{x \in X_1} d(x, \chi(x) )  + \sum_{y \in X_2 \cup \{ b \} \setminus \chi(X_1)} M(y) \\
&=& d(a,b) +  \sum_{x \in X_1 } d(x, \eta(x) )   - d(a,c) +  \sum_{y \in X_2 \setminus \eta(X_1)   } M(y)   + M(c) \\
&=& d(a,b) +M(c) - d(a,c)  +  \sum_{x \in X_1} d(x, \eta(x) )   + \sum_{y \in X_2 \setminus \eta(X_1)  } M(y) \\
&=& d(a,b) +M(c) - d(a,c) + d_\eta(X_1, X_2) \\
&\ge& d_\eta(X_1, X_2) \ge d_S(X_1, X_2)
\end{eqnarray*}
since $d(a,c) \le M(c)$ by condition (\ref{condition}).
\end{proof}
By repeated application of the above result, we obtain the following corollary. 
\begin{cor} \label{grow}
For any $X_1, X_2 \in \mathcal{F}(X)$, such that $\vert X_1\vert  \leq \vert X_2\vert $. Suppose that $X_2' \subseteq X_2$ such that $\vert X_1\vert  \leq \vert X_2'\vert $. Then $$d_S(X_1,X_2') \leq d_S(X_1,X_2).$$ 
\end{cor}

%

\begin{theorem} \label{metric}
 $d_S(\cdot,\cdot)$ is a metric on $\mathcal{F}(X)$.
\end{theorem}
\begin{proof}
For two finite sets $A$ and $B$ we denote by $\mathscr{X}(A,B)$ the set of injections $\chi: A\to B$. 
Let $X_1,X_2 \in \mathcal{F}(X)$.  By definition of $d_S(\cdot,\cdot)$  we have that $d_S(X_1,X_2) = d_S(X_2,X_1) \geq 0$. 
We show that $d_S(X_1, X_2)=0$ iff $X_1=X_2$.  We may assume that $\vert X_1\vert  \leq \vert X_2\vert $, and let $\nu  \in \mathscr{X}(X_1,X_2)$ be such that $d_S(X_1,X_2) = d_{\nu }(X_1,X_2)$. Then
$d_S(X_1,X_2) = d_{\nu }(X_1,X_2) = 0$ 
iff 
$\displaystyle\sum_{x \in X_1}d(x,\nu (x)) + \sum_{y\in X_2\setminus \nu(X_1) } M(y) = 0$ 
iff 
$d(x,\nu (x)) = 0$ for all  $x \in X_1 \text{ and } X_2= \nu(X_1) $
 iff 
 $x = \nu (x)$  for all $x \in X_1 \text{ and } \vert X_2\vert  = \vert X_1\vert $ 
 iff
  $X_1 = X_2$. 

 Thus, we need only to show that $d_S(\cdot,\cdot)$ satisfies the Triangle Inequality. 
 Let $X_1, X_2, X_3\in \mathcal{F}(X)$. 
 We will show that $d_S(X_1, X_2) \le d_S(X_1, X_3) + d_S(X_3,X_2)$ by considering various cases.
Note that we are still assuming that  $\vert X_1\vert  \leq \vert X_2\vert $, and that  $\nu  \in \mathscr{X}(X_1,X_2)$ is  such that $d_S(X_1,X_2) = d_{\nu }(X_1,X_2)$. 

\textbf{Case 1:} Suppose that $\vert X_1\vert  \leq \vert X_3\vert  \leq \vert X_2\vert $. 
 Let 
$\mu \in \mathscr{X}(X_3,X_2)$ and 
$\eta \in \mathscr{X}(X_1,X_3)$, be such that $d_S(X_3,X_2) = d_{\mu}(X_3,X_2)$ and $d_S(X_1,X_3) = d_{\eta}(X_1,X_3)$. 
We may assume that 
\begin{align*}
X_1 =& \{ x_1, \ldots, x_n \} \\
X_3 =& \{ y_1, \ldots, y_n, y_{n+1}, \ldots, y_{n+s} \} \\
X_2 =& \{ z_1, \ldots, z_n, z_{n+1}, \ldots, z_{n+s}, \ldots, z_{n+s+t} \} 
\end{align*}
 where $s,t\ge 0$ and $\mu(y_i) = z_i$ for $1\le i \le n+s$ and  
$\eta(x_i) = z_i$ for $1\le i \le n$. Then 
\begin{align*}
d_S(X_1,X_3)  =&  \sum_{i=1}^n d(x_i, y_i)  + \sum_{i=n+1}^{n+s} M(y_i) \hbox{ and } \\
d_S(X_2,X_3)  =&  \sum_{i=1}^{n+s} d(y_i, z_i)  + \sum_{i=n+s+1}^{n+s+t} M(z_i). 
\end{align*}
Let $\chi = \mu \circ \eta \in  \mathscr{X}(X_1,X_2)$.
Then
\begin{align*}
& d_S(X_1, X_2)  \\
\le & d_\chi(X_1, X_2) \\
=& \sum_{i=1}^n  d(x_i, z_i) + \sum_{i=n+1}^{n+s+t} M(z_i) \\
\le& \sum_{i=1}^n  \left [ d(x_i, y_i) + d(y_i, z_i) \right ] + \sum_{i=n+1}^{n+s+t} M(z_i) \\
= & \sum_{i=1}^n  d(x_i, y_i) +\sum_{i=1}^n  d(y_i, z_i) + \sum_{i=n+1}^{n+s+t} M(z_i) \\
=& \left ( d_S(X_1, X_3)  - \sum_{i=n+1}^{n+s} M(y_i)  \right )   + \\
 &  \left(  d_S(X_3, X_2) - \sum_{i=n+1}^{n+s} d(y_i, z_i)  - \sum_{i=n+s+1}^{n+s+t} M(z_i) \right )   +  \sum_{i=n+1}^{n+s+t} M(z_i) \\
=& d_S(X_1,X_3) + d_S(X_2, X_3)    -  \sum_{i=n+1}^{n+s} M(y_i) - \sum_{i=n+1}^{n+s} d(y_i, z_i) +
 \sum_{i=n+1}^{n+s} M(z_i) \\
=& d_S(X_1,X_3) + d_S(X_2, X_3)    + \sum_{i=n+1}^{n+s} \left ( M(z_i) - M(y_i) -  d(y_i, z_i)  \right ) \\
\le & d_S(X_1, X_3) + d_S(X_3,X_2).
\end{align*}
by condition (\ref{condition})

\textbf{Case 2:} Suppose $\vert X_3\vert  \leq \vert X_1\vert  \leq \vert X_2\vert $. Let 
$\mu \in \mathscr{X}(X_3,X_2)$ and 
$\eta \in \mathscr{X}(X_3,X_1)$ be such that $d_S(X_3,X_2) = d_{\mu}(X_3,X_2)$ and $d_S(X_1,X_3) = d_{\eta}(X_1,X_3)$. 
We may assume that 
\begin{eqnarray*}
X_3 &=& \{ x_1, \ldots, x_n \} \\
X_1 &=& \{ y_1, \ldots, y_n, y_{n+1}, \ldots, y_{n+s} \} \\
X_2 &=& \{ z_1, \ldots, z_n, z_{n+1}, \ldots, z_{n+s}, \ldots, z_{n+s+t} \} 
\end{eqnarray*}
 where $s,t\ge 0$ and $\mu(x_i) = z_i$ for $1\le i \le n$ and  
$\eta(x_i) = y_i$ for $1\le i \le n$. Then 
\begin{align*}
d_S(X_3,X_1)  =&  \sum_{i=1}^n d(x_i, y_i)  + \sum_{i=n+1}^{n+s} M(y_i) \\
d_S(X_3,X_2)  =&  \sum_{i=1}^{n} d(x_i, z_i)  + \sum_{i=n+1}^{n+s+t} M(z_i). 
\end{align*}
Define $\chi: X_1 \to X_2$ by  $\chi(y_i) = z_i$ for $i=1,2,\ldots, n+s$. 
Then
\begin{align*}
& d_S(X_1, X_2)  \\
\le & d_\chi(X_1, X_2) \\
=& \sum_{i=1}^{n+s}  d(y_i, z_i) + \sum_{i=n+s+1}^{n+s+t} M(z_i) \\
=& \sum_{i=1}^{n}  d(y_i, z_i) +\sum_{i=n+1}^{n+s}  d(y_i, z_i)+  \sum_{i=n+s+1}^{n+s+t} M(z_i)  \\
\le& \sum_{i=1}^{n}  \left [ d(y_i, x_i) + d(x_i, z_i) \right ]    +\sum_{i=n+1}^{n+s}  d(y_i, z_i)+  \sum_{i=n+s+1}^{n+s+t} M(z_i)   \\
= & \sum_{i=1}^{n} d(y_i, x_i) + \left ( \sum_{i=1}^{n}d(x_i, z_i)  +\sum_{i=n+1}^{n+s+t} M(z_i)  \right) -   \sum_{i=n+1}^{n+s} M(z_i)   +\sum_{i=n+1}^{n+s}  d(y_i, z_i)   \\
=&    \left (  d_S(X_3,X_1) - \sum_{i=n+1}^{n+s} M(y_i)  \right ) + 
d_S(X_3,X_2) + 
 \sum_{i=n+1}^{n+s}  \left ( d(y_i, z_i) - M(z_i)  \right) \\
 =&   d_S(X_3,X_1)+  d_S(X_3,X_2)   - \sum_{i=n+1}^{n+s} M(y_i) + 
 \sum_{i=n+1}^{n+s}  \left ( d(y_i, z_i) - M(z_i)  \right)   \\
\le & d_S(X_1, X_3) + d_S(X_3,X_2).
\end{align*}
where the last inequality follows from by condition (\ref{condition}).

\textbf{Case 3:} Suppose $\vert X_1\vert  \leq \vert X_2\vert  \le  \vert X_3\vert $. 

Fix a subset $X_3'$ of $X_3$ of cardinality equal to $X_2$. 
Then from Case 1, it follows that $d_S(X_1, X_2) \le d_S(X_1,X_3') + d_S(X_3', X_2)$. From Corollary \ref{grow} we know that 
$d_S(X_1, X_3') \le d_S(X_1, X_3)$ and  $ d_S(X_3', X_2) \le d_S(X_3, X_2)$. Thus
 $d_S(X_1, X_2) \le d_S(X_1,X_3) + d_S(X_3, X_2)$.

\end{proof}

\begin{remark} \label{remark1}
If $X$ contains at least two elements, then the function $M$ never takes on the value 0. In fact, there exists a constant $C>0$ such that 
$M(y) \ge C$ for all $y\in X$: from (\ref{condition}), 
$d(x,y) \le M(x) \le  d(x,y) + M(y) \le   2M(y)$. Thus $M(y) \ge M(x)/2$ for all $y\in X$.  Put $C = M(x)/2$. If $C=0$, then the inequlity 
$M(y) \ge d(y,x)$ implies that $y= x$ for all $x\in X$, contradicting that $X$ contains at least two elements. Thus $C = M(x)/2>0$ is the required constant. 
\end{remark}
\begin{remark}\label{remark2}
If  $\{ A_n \}$ is a Cauchy sequence in $\mathcal{F}(X)$, it can be shown that $\vert A_n\vert  = \vert A_m\vert $ for all $m, n$ sufficiently large: let $C$ be as in Remark \ref{remark1}. 
Then there exists $N$ such that $d_S(A_m,A_n) <  C$ for all $m, n \ge N$.  Since $ C =  \frac{1}{2} M(x)  < M(y)$ for all $y\in X$, it follows that 
$\vert A_m\vert  = \vert A_n\vert $ for all $m, n \ge N$.
\end{remark}  
\begin{remark}
If the topology induced the metric $d$ on $X$ is the discrete topology, then $\mathcal{F}(X)$ is complete with respect to  the subset metric. 
However, this is not the case in general.  Consider the case where $X = [0,1]$, $d$ is the usual Euclidean metric and $M(y) =\max\{ y, 1-y \}$. 
Put   $A_n = \{ 0, \frac{1}{n} \}$ for all $n\ge 1$. 
Then $\{A_n\}$ is Cauchy sequence that does not converge: if $\{ A_n \}$ did converge, using Lemma \ref{lemma1} and Remark \ref{remark2},
 it would converge to a set of the form 
$A = \{ 0, a\}$ for some $a\in X$. But $d_S(A_n, A) = \vert a-1/n\vert  \to 0 $ as $n\to \infty$, so $a$ must equal to $0$. But if $a=0$, then
$d_(A_n, A) = M(1/n) = 1-1/n \to 1$ as $n\to\infty$, a contradiction. \end{remark}


\begin{thebibliography}{1}

\bibitem{conci}
Aura Conci and Carlos Kubrusly.
\newblock Distances between sets---a survey.
\newblock {\em Adv. Math. Sci. Appl.}, 26(1):1--18, 2017.

\bibitem{ency}
Michel~Marie Deza and Elena Deza.
\newblock {\em Encyclopedia of distances}.
\newblock Springer-Verlag, Berlin, 2009.
\newblock With 1 CD-ROM (Windows, Macintosh and UNIX).

\bibitem{EiterMannila}
Thomas Eiter and Heikki Mannila.
\newblock Distance measures for point sets and their computation.
\newblock {\em Acta Inform.}, 34(2):109--133, 1997.

\bibitem{munkres}
James Munkres.
\newblock Algorithms for the assignment and transportation problems.
\newblock {\em J. Soc. Indust. Appl. Math.}, 5:32--38, 1957.

\bibitem{Niini}
Ilkka Niiniluoto.
\newblock {\em Truthlikeness}, volume 185 of {\em Synthese Library.}
\newblock Springer Dordrecht, 1987.

\bibitem{Oddie}
Ilkka Niiniluoto and Raimo Tuomela, editors.
\newblock {\em The logic and epistemology of scientific change}. Societas
  Philosophica Fennica, Helsinki, 1979.
\newblock Acta Philos. Fenn. {{\bf{3}}0} (1978), no. 2-4 (1979).

\bibitem{cai}
Wentu Song, Kui Cai, and Kees~A. Schouhamer~Immink.
\newblock Sequence-subset distance and coding for error control in {DNA}-based
  data storage.
\newblock {\em IEEE Trans. Inform. Theory}, 66(10):6048--6065, 2020.

\end{thebibliography}


\end{document}